\documentclass{article}
\usepackage{arxiv}

\usepackage[utf8]{inputenc} % allow utf-8 input
\usepackage[T1]{fontenc}    % use 8-bit T1 fonts
\usepackage{hyperref}       % hyperlinks
\usepackage{url}            % simple URL typesetting
\usepackage{booktabs}       % professional-quality tables
\usepackage{amsfonts}       % blackboard math symbols
\usepackage{nicefrac}       % compact symbols for 1/2, etc.
\usepackage{microtype}      % microtypography
\usepackage{lipsum}		% Can be removed after putting your text content
\usepackage{graphicx}
\usepackage[square,numbers]{natbib}
\usepackage{doi}

\usepackage{amsmath, amssymb, amsthm}
\usepackage{enumitem}
\theoremstyle{theorem}
\newtheorem{theorem}{Theorem}
\newtheorem{proposition}{Proposition}

\theoremstyle{remark}
\newtheorem{remark}[theorem]{Remark}

\title{\textsf{Catastrophic failure and cumulative damage models involving two types of extended exponential distributions}}

\author{
  Hiroaki Mohri\\
  Faculty of Commerce,
  Waseda University\\
  Shinjuku-ku, Tokyo 169-8050, Japan\\
  \texttt{mohri@waseda.jp}\\
	\And
	Jun-ichi Takeshita \\
        Research Institute of Science for Safety and Sustainability,\\
        National Institute of Advanced Industrial Science and Technology (AIST)\\
        Tsukuba, Ibaraki 305-8569, Japan\\
        \texttt{jun-takeshita@aist.go.jp}\\
}

\renewcommand{\headeright}{H. Mohri and J. Takeshita}
\renewcommand{\undertitle}{}
\renewcommand{\shorttitle}{Damage analyses involving two stochastic processes}

\usepackage{autonum}

\begin{document}
\maketitle

\begin{abstract}
The present study supposes a single unit and investigates cumulative damage and catastrophic failure models for the unit, in situations where the interarrival times between the shocks, and the magnitudes of the shocks, involve two different stochastic processes.
  In order to consider two essentially different stochastic processes, integer gamma and Weibull distributions are treated as distributions with two parameters and extensions of exponential distributions.
  With respect to the cumulative damage models, under the assumption that the interarrival times between shocks follow exponential distributions, the case in which the magnitudes of the shocks follow integer gamma distributions is analyzed.
  With respect to the catastrophic failure models, the respective cases in which the interarrival times between shocks follow integer gamma and Weibull distributions are discussed.
  Finally, the study provides some characteristic values for reliability in such models.
\end{abstract}


\section{Introduction}
\label{sec:introduction}

In the literature of reliability theory, numerous studies have concentrated on failure models; see e.g., Barlow and Proschan~\cite{BP1996}, Nakagawa~\cite{Na2007}.
In some cases, it is enough to only know when the failure events occur, but in other cases it is necessary to know the timing of, and damage caused by, the events.
In these latter cases, we utilize the models for two-phase failure analysis; in the first phase, focusing solely the stochastic process involved in the failure events, and in the second, considering both the stochastic process and the probabilistic distribution of the magnitude of damage in the events.

As a simple concrete example, consider the case of house collapse due to an earthquake.
If the house collapses due to a massive earthquake, and we thereby need not consider additional shocks from further quakes, first-phase analysis alone may be sufficient.
In this case, a shock of sufficient magnitude is applied to the unit (= house) such that it immediately fails. However, should the house collapse as a result of successive lesser earthquakes, the damage magnitudes of the respective earthquakes are additive.
When the shocks are additive and the total shock magnitude exceeds a given number $K>0$, the unit fails. This case illustrates a cumulative damage model, which has been employed in numerous studies, including that of Cox~\cite{Cox1962}.

There are two types of earthquake: oceanic trench and inland.
Each type of shock involves a different stochastic process, and the respective magnitudes of each shock type show different probabilistic distributions.
Given the frequent occurrence of both types of failure event, we often cannot apply previous studies to analyze realistic situations.
Note that since exponential distributions have the reproductive property, two exponential distributions are essentially equivalent to one exponential distribution.

The present study supposes that both types of shock are applied to a single unit, and deals with integer gamma distributions and Weibull distributions, which include exponential distributions as special cases. Even if the shock types differ, we investigate both the catastrophic (i.e., no need to consider the shock magnitude) and cumulative shock types, and summarize our models as follows.
\begin{description}
\item[Model 1] (catastrophic shock): When a shock is applied to a unit, it immediately fails.
\item[Model 2] (cumulative shock): When the total magnitude for a unit exceeds a given number $K>0$, it fails. Note that the total magnitude for the unit is the sum of all shock types.
\end{description}
For Model 1, the following will be derived: (i) The probability that the unit does not fail at time t; (ii) the first passage of time to failure (time to initial failure, hereinafter referred to as FPTF), and its mean.
For Model 2, the following will be derived: (i) The distribution function of the sum of damage and its mean; (ii) the FPTF, which is equal to the probability of the total shock being greater than $K$.

The paper is organized as follows.
Section 2 summarizes the basic notions and assumptions of the study.
Section 3 describes the results for both Model 1 and Model 2 in the case of general distributions.
Section 4 deals with Model 2 when the magnitudes of the events follow integer gamma distributions, under the assumption that the interarrival times between shocks follow exponential distributions.
Section 5 deals with Model 1 when the interarrival times between shocks follow integer gamma distributions and Weibull distributions.
Section 6 concludes the study.

\section{Preliminaries}
\label{ec:pre}

The study employed the following notions and assumptions.
\begin{itemize}
 \item A unit has generation intervals determined by stochastic processes 1 and 2, say, $X_i$ and $Y_i$ $(i \in \mathbb{N})$, respectively.
       It is assumed that $X_i$ and $Y_i$ are independent and identically distributed random variables, and have identical general distributions of $F_{X_i}(t):= P(X_i \leq t)$ and $F_{Y_i}(t):= P(Y_i \leq t)$, respectively.
 \item The respective weights of the $i$-th event, based on stochastic processes 1 and 2, are denoted as $W_i$ and $V_i$, respectively.
       It is assumed that $W_i$ and $V_i$ are independent and identically distributed random variables, and have identical general distributions of $G_{W_i}(x):= P(W_i \leq x)$ and $G_{V_i}(x):= P(V_i \leq x)$, respectively.
 \item The numbers of events in $[0, t] \ (t>0)$, based on stochastic processes 1 and 2, are denoted as $N_1(t)$ and $N_2(t)$, respectively.
       It is assumed that $N_1(t)$ and $N_2(t)$ are independent.
 \item $\mathcal{W}(t) := \sum_{i=1}^{N_1(t)} W_i$, $\mathcal{V}(t) := \sum_{i=1}^{N_2(t)} V_i$, and $U(t):= \mathcal{W}(t) + \mathcal{V}(t)$.
 \item For any distribution function $H(t, \cdot)$, its Laplace-Stieltjes (LS) transform with respect to $t$ is denoted as $\mathcal{L}_t(H(\cdot))[s]$, and its inverse LS transform is denoted as $\mathcal{L}^{-1}_{s}$.
 \end{itemize}

 \section{Case of general distributions}
\label{sec:case_gnrl}

In the case of Model 1 for general distributions, we can easily obtain the following.
\begin{proposition}
  Assume that $F_{X_i} = F_X$, $F_{Y_i} = F_Y$, $G_{W_i}=G_W$, $G_{V_i}=G_V$ for any $i$, and
  $E(X_i) = E(X) < \infty$, $E(Y_i) = E(Y) < \infty$ for any $i$.
  Then the following hold:
  \begin{enumerate}[label=(\roman*)]
  \item The probability that the unit does not fail at time $t$ is $(1 - F_X (t)) (1 - F_Y (t))$;
  \item FPTF is $F_X (t) + F_Y (t) - F_X (t) F_Y(t)$; 
  \item The mean of FPTF is given by
    \begin{equation}
      \label{eq:1}
      \int_0^{\infty} (1 - F_X (t)) (1 - F_Y (t)) \, dt.
    \end{equation}
  \end{enumerate}
\end{proposition}

In the case of Model 2 for general distributions, Mohri and Takeshita~\cite{MT2018} showed the following.
\begin{proposition}(Section 3 in Mohri and Takeshita~\cite{MT2018})
  \label{prop:model2_gnrl}
  Assume that $F_{X_i} = F_X$, $F_{Y_i} = F_Y$, $G_{W_i}=G_W$, and $G_{V_i}=G_V$ for any $i$.
  Then the following hold:
  \begin{enumerate}[label=(\roman*)]
  \item The distribution function of the sum of damages is given by 
    \begin{multline}
      P(U(t) \leq x) =    \sum_{k_1=0}^{\infty} \sum_{k_2=0}^{\infty}
                   \left( G^{(k_1)}_W \ast G^{(k_2)}_V\right) (x)
   \left(F^{(k_1)}_X (t) - F^{(k_1 + 1)}_X (t)\right)
     \left(F^{(k_2)}_Y (t) - F^{(k_2 + 1)}_Y (t)\right);  
   \end{multline}
 \item The mean of (i) is given by 
   \begin{multline}
       E(U) = -  \left[\sum_{k_1=0}^{\infty} k_1
  \left.\frac{d\mathcal{L}_x (G_W)[s]}{ds}\right|_{s=0}  \left(F_X^{(k_1)}(t) - F_X^{(k_1+1)}(t)\right)\right]\\
 -  \left[\sum_{k_2=0}^{\infty} k_2 
   \left.\frac{d \mathcal{L}_x (G_V)[s]}{ds}\right|_{s=0}   \left(F_Y^{(k_2)}(t) - F_Y^{(k_2+1)}(t)\right)\right],
   \end{multline}
 \item FPTF is given by 
   \begin{multline}
      P(Y \leq t) = 1 - \sum_{k_1=0}^{\infty} \sum_{k_2=0}^{\infty}
  \left(G^{(k_1)}_W * G^{(k_2)}_V\right)(K)
   \left(F_X^{(k_1)}(t) - F_X^{(k_1+1)}(t)\right)
  \left(F_Y^{(k_2)}(t) - F_Y^{(k_2+1)}(t)\right),
\end{multline}
%\item The mean of FPTF is given by 
%  \begin{multline}
%    E(Y)
%   = \sum_{k_1=0}^{\infty} \sum_{k_2=0}^{\infty} \left(G^{(k_1)}_W * G^{(k_2)}_V\right)(K)\\
%%   \times \frac{d}{ds'} \bigg(\mathcal{L}_t \Big[\left(F_X^{(k_1)}(t) - F_X^{(k_1+1)}(t)\right)\\
 % \times \left.\left(F_Y^{(k_2)}(t) - F_Y^{(k_2+1)}(t)\right)\Big][s']\bigg)
 %\right|_{s'=0} 
 % \end{multline}
\end{enumerate}
\end{proposition}

\section{Case that the interarrival times between shocks (i.e., $X_i$ and $Y_i$) follow exponential distributions}
\label{sec:interval_case_exp}
 In this section, we assume that $X_i$ and $Y_i$ follow exponential distributions.

 \subsection{Case that the magnitudes of the events (i.e., $W_i$ and $V_i$) follow exponential distributions}

  Model 2 can be applied, with the assumption that $\mu_1 = \mu_2 := \mu$, [i.e., the existing results for damage models (e.g., Nakagawa and Osaki~\cite{NO1974}, Cox3\cite{Cox1962})], because exponential distributions have the reproductive property.
  More precisely, we should analyze a new random variable $\{X'_i\}$ using the distribution function $F_{X'_i}(t):= 1 - \exp (- (\lambda_1 + \lambda_2) t)$, instead of $\{X_i\}$ and $\{Y_i\}$; and analyze $\{W'_i\}$ using the distribution function $G_{W'_i}(x) := 1 - \exp (- \mu x)$, instead of $\{W_i\}$ and $\{V_i\}$.

   \subsection{Case that magnitudes of events (i.e., $W_i$ and $V_i$) follow integer gamma distributions}

In this subsection, we analyze Model 2 for the cases where $W_i$ and $V_i$ follow integer gamma distributions.
Note that exponential distributions are special cases of integer gamma distributions; see Remark~\ref{rem:case_m=1}.
In general, $\lambda_1 \neq \lambda_2 \neq \mu_1 \neq \mu_2$ and $m_1, m_2 \in \mathbb{N}$, the distribution function of the sum of damage, its means, and the FPTF are given by the following.
\begin{proposition}
  \label{thm:MT2018_exp_mag_i_gamma}
  Assume that $F_{X_i}(t) = 1 - \exp (- \lambda_1 t)$, $F_{Y_i}(t) = 1 - \exp (- \lambda_2 t)$,
  \begin{equation}
    \label{eq:2}
    G_{W_i}(x) = 1 - \exp (- \mu_1 x) \sum_{l_1=0}^{m_1-1} \frac{(\mu_1 x)^{l_1}}{l_1!} \quad \text{, and} \quad 
   G_{V_i}(x) = 1 - \exp (- \mu_2 x) \sum_{l_2=0}^{m_2-1} \frac{(\mu_2 x)^{l_2}}{l_2!},  
  \end{equation}
  for any $i$, where $\lambda_1, \lambda_2, \mu_1, \mu_2 >0$ and generally $\lambda_1 \neq \lambda_2 \neq \mu_1 \neq \mu_2$.
  Also, we let $\lambda:= \lambda_1 + \lambda_2$ and $m:= m_1 + m_2$.
  Then the following hold:
    \begin{enumerate}[label=(\roman*)]
  \item The distribution function of the sum of damage is given by 
    \begin{multline}
         P(U(t) \leq x) = \sum_{k_1=0}^{\infty} \sum_{k_2=0}^{\infty}
    \mathcal{L}_s^{-1} \left[   \left(\frac{\mu_1}{s + \mu_1}\right)^{m_1 k_1}
		      \left(\frac{\mu_2}{s + \mu_2}\right)^{m_2 k_2}\right](x)
   \frac{(\lambda_1 t)^{k_1}}{k_1 !}
     \frac{(\lambda_2 t)^{k_2}}{k_2 !} e^{-\lambda t};
   \end{multline}
 \item The mean of (i) is given by $E(U(t)) = m_1 \lambda_1 t / \mu_1 + m_2 \lambda_2 t / \mu_2$;
 \item FPTF is given by 
   \begin{multline}
         P(Y \leq t) = P(U(t) > K)\\
         = 1 -  \sum_{k_1=0}^{\infty} \sum_{k_2=0}^{\infty}
    \mathcal{L}_s^{-1} \left[   \left(\frac{\mu_1}{s + \mu_1}\right)^{m_1 k_1}
      \left(\frac{\mu_2}{s + \mu_2}\right)^{m_2 k_2}\right](K)
   \frac{(\lambda_1 t)^{k_1}}{k_1 !}
     \frac{(\lambda_2 t)^{k_2}}{k_2 !} e^{-\lambda t}.
\end{multline}
\end{enumerate}
\end{proposition}

\begin{proof}
  Since (iii) can be immediately obtained from (i), we prove only (i) and (ii).

  From 
  \begin{equation}
    \mathcal{L}_x\left( G_{W_i} \right)[s] = \left(\frac{\mu_1}{s + \mu_1}\right)^{m_1} \quad \text{and} \quad
    \mathcal{L}_x\left( G_{V_i} \right)[s] = \left(\frac{\mu_2}{s + \mu_2}\right)^{m_2}
  \end{equation} 
  and Proposition~\ref{prop:model2_gnrl} (i), we obtain
 	\begin{multline}
         P(U(t) \leq x) = \sum_{k_1=0}^{\infty} \sum_{k_2=0}^{\infty}
    \mathcal{L}_s^{-1} \left[   \left(\frac{\mu_1}{s + \mu_1}\right)^{m_1 k_1}
		      \left(\frac{\mu_2}{s + \mu_2}\right)^{m_2 k_2}\right](x)
   \frac{(\lambda_1 t)^{k_1}}{k_1 !}
     \frac{(\lambda_2 t)^{k_2}}{k_2 !} e^{-\lambda t}.
   \end{multline}
   Also, from (i), we have
  \begin{align}
    E(U(t)) &= - \left.\frac{d \mathcal{L}_x \left( P(U(t) < x) \right)[s]}{ds}\right|_{s=0}\notag\\
            &= \sum_{k_1=0}^{\infty} \sum_{k_2=0}^{\infty}
              \left(\frac{\mu_1 k_1}{s + \mu_1} + \frac{\mu_2 k_2}{s + \mu_2}\right)
              \left.\left(\frac{\mu_1}{s + \mu_1}\right)^{m_1 k_1} \left(\frac{\mu_2}{s + \mu_2}\right)^{m_2 k_2}\right|_{s=0}
             \frac{(\lambda_1 t)^{k_1}}{k_1 !} \frac{(\lambda_2 t)^{k_2}}{k_2 !} e^{-\lambda t}\notag\\
            &= \sum_{k_1=0}^{\infty} \sum_{k_2=0}^{\infty} \left(\frac{m_1 k_1}{\mu_1} + \frac{m_2 k_2}{\mu_2}\right)
              \frac{(\lambda_1 t)^{k_1}}{k_1 !} \frac{(\lambda_2 t)^{k_2}}{k_2 !} e^{-\lambda t}\notag\\
            &=\left(\frac{m_1 \lambda_1 t}{\mu_1}\sum_{k_1=1}^{\infty}\frac{(\lambda_1 t)^{k_1 - 1}}{(k_1-1)!}
              \sum_{k_2=0}^{\infty}\frac{(\lambda_2 t)^{k_2}}{k_2 !}
            +\frac{m_2 \lambda_2 t}{\mu_2}\sum_{k_1=0}^{\infty}\frac{(\lambda_1 t)^{k_1}}{k_1!}
              \sum_{k_2=1}^{\infty}\frac{(\lambda_2 t)^{k_2 - 1}}{(k_2 -1) !} \right) e^{-\lambda t}\notag\\
            &= \frac{m_1 \lambda_1 t}{\mu_1} + \frac{m_2 \lambda_2 t}{\mu_2},
  \end{align}
  which is the conclusion of (ii).
\end{proof}

\begin{remark}
  \label{rem:case_m=1}
  When $m_1 = m_2 =1$, Proposition~\ref{thm:MT2018_exp_mag_i_gamma} shows the results for exponential distributions.
\end{remark}

\begin{remark}
  Proposition~\ref{thm:MT2018_exp_mag_i_gamma} (i) includes the inverse LS transform and does not provides explicit formula.
  Although the gamma distributions do not have the reproductive property when $\mu_1 \neq \mu_2$, Coelho~\cite{Coelho1998} gave the probability density function of sum of integer gamma distributions for $\mu_1 \neq \mu_2$; and using Corollary 1 in Coelho~\cite{Coelho1998}, we can obtain an explicit formula of (i).
\end{remark}

  \section{Case that magnitudes of events can be ignored}

In this section, we analyze Model 1 for the case where $X_i$ and $Y_i$ follow integer gamma and Weibull distributions.
Note that this section this section is not concerned with the distributions followed by $V_i$ and $W_i$.

\subsection{Case that interarrival times between shocks (i.e., $X_i$ and $Y_i$) follow integer gamma distributions}
\label{sec:case3}

In this subsection, we analyze Model 1 for the cases where $W_i$ and $V_i$ follow integer gamma distributions.

\begin{proposition}
  Assume that
  \begin{equation}
      F_{W_i}(t) = 1 - \exp (- \lambda_1 t) \sum_{l_1=0}^{m_1-1} \frac{(\lambda_1 x)^{l_1}}{l_1!} \quad \text{, and} \quad
      F_{V_i}(t) = 1 - \exp (- \lambda_2 t) \sum_{l_2=0}^{m_2-1} \frac{(\lambda_1 x)^{l_2}}{l_2!},
  \end{equation}
  where $k_1, k_2 \in \mathbb{N}$, $\lambda_1, \lambda_2 > 0$ and generally $m_1 \neq m_2$ and $\lambda_1 \neq \lambda_2$.
  Also we let $\lambda := \lambda_1 + \lambda_2$.
  Then the following hold:
  \begin{enumerate}[label=(\roman*)]
  \item The probability that the unit does not fail at time $t$ is given by
    \begin{equation}
      \exp (- \lambda t )
      \left(\sum_{l_1=0}^{m_1-1} \frac{(\lambda_1 t)^{l_1}}{l_1 !}\right)
      \left(\sum_{l_2=0}^{m_2-1} \frac{(\lambda_2 t)^{l_2}}{l_2 !}\right);
    \end{equation}
  \item FPTF is given by
    \begin{equation}
      1 - \exp (- \lambda t )
      \left(\sum_{l_1=0}^{m_1-1} \frac{(\lambda_1 t)^{l_1}}{l_1 !}\right)
      \left(\sum_{l_2=0}^{m_2-1} \frac{(\lambda_2 t)^{l_2}}{l_2 !}\right);
    \end{equation}
  \item The means of FTPF are given by
    \begin{equation}
      \left\{
        \begin{array}{ll}
          \displaystyle \sum_{l_1=0}^{m_1-1} \frac{\lambda_1^{l_1}}{\lambda^{l_1+1}}& \text{if } m_2 = 1,\\[4mm]
          \displaystyle \sum_{n=0}^{2 (m-1)} \sum_{l=0}^n \binom{n}{l}^2
          \frac{\lambda_1^{n-l} \lambda_2^l}{\lambda^{n+1}}& \text{if } m:= m_1 = m_2.\\
        \end{array}
      \right.
    \end{equation}
  \end{enumerate}
\end{proposition}

\begin{proof}
  Since $P(Z(t)=0)$ expresses the probability that the unit does not fail at time $t$ and $P(Z(t) = 0) = P(X_1 >t)P(Y_1 >t)$, we should derive $P(X_1 >t)$ and $P(Y_1 > t)$.
  From
  \begin{align}
    P(X_1  \leq t) &= 1 - \exp (- \lambda_1 t) \sum_{l_1=0}^{m_1-1} \frac{(\lambda_1 t)^{l_1}}{l_1!},\\
        P(Y_1  \leq t) &= 1 - \exp (- \lambda_2 t) \sum_{l_2=0}^{m_2-1} \frac{(\lambda_2 t)^{l_2}}{l_2!},
  \end{align}
  we obtain (i); and (ii) can be immediately obtained from (i).

  Let $T$ be FPTF, that is, $T:= \min\{t \, | \, Z(t) > 0\}$, and $\Phi(t):=P(T \leq t)$.
  Then $E(T) = - d\Phi^*[s] / ds |_{s=0}$, where $\Phi^*$ represents the Laplace-Steitltjes transform of $\Phi$.

  When $m_2=1$, we have
  \begin{align}
    \Phi(t) &= 1 - \sum_{l_1=0}^{m_1-1} \left(\frac{\lambda_1^{l_1}}{l_1!} t^{l_1} e^{-\lambda t}\right),\\
    \Phi^*(s) &= 1 - \sum_{l_1=0}^{m_1-1} \left(\frac{\lambda_1^{l_1}}{l_1!} \frac{s l_1 !}{(s + \lambda)^{l_1+1}}\right).
  \end{align}
  Then we obtain
  \begin{equation}
    E(T) = \left.\sum_{l_1=0}^{m_1-1} \left(\frac{\lambda_1^{l_1}}{(s + \lambda)^{l_1+1}}
      - \frac{(l_1+1) \lambda_1^{l_1} s}{(s + \lambda)^{l_1+2}}\right) \right|_{s=0}
  = \sum_{l_1 = 0}^{m_1 - 1} \frac{\lambda_1^{l_1}}{\lambda^{l_1+1}}.
\end{equation}

  On the other hand, when $m_1=m_2=:m$, we get
  \begin{align}
    \Phi(t) &= 1 - \left(\sum_{l_1=0}^{m_1-1}\frac{\lambda_1^{l_1}}{l_1!} t^{l_1}\right)
              \left(\sum_{l_2=0}^{m_2-1}\frac{\lambda_2^{l_2}}{l_2!} t^{l_2}\right) e^{-\lambda t}\notag\\
            &= 1 - \sum_{n=0}^{2(m-1)} \sum_{l=0}^n
              \left( \binom{n}{l}^2 \frac{\lambda_1^{n-l} \lambda_2^l}{(n-l)! l!} t^n e^{-\lambda t} \right),\\
    \Phi^*(s) &= 1 - \sum_{n=0}^{2(m-1)} \sum_{l=0}^n
              \left( \binom{n}{l}^2 \frac{\lambda_1^{n-l} \lambda_2^l}{(n-l)! l!} \frac{n!}{(s+\lambda)^{n+1}} \right).
  \end{align}
   Then we obtain
  \begin{align}
    E(T) &= \left. \sum_{n=0}^{2(m-1)} \sum_{l=0}^n \binom{n}{l}^2 \lambda_1^{n-l} \lambda_2^l
           \left(\frac{1}{(s + \lambda)^{n+1}} - \frac{(n+1)s}{(s + \lambda)^{n}}\right) \right|_{s=0}\notag\\
    &= \sum_{n=0}^{2(m-1)} \sum_{l=0}^n \binom{n}{l}^2
      \frac{\lambda_1^{n-l} \lambda_2^l}{\lambda^{n+1}}.
  \end{align}
  \mbox{}
\end{proof}

\subsection{Case that interarrival times between shocks (i.e., $X_i$ and $Y_i$) follow Weibull distributions}

\begin{proposition}
  \label{prop:event_weibull}
  Assume that  $F_{X_i}(t):= 1 - \exp (- (t / \beta_1)^{\alpha_1} )$ and $F_{Y_i}(t):= 1 - \exp (- (t / \beta_2)^{\alpha_2} )$, where $\alpha_1, \alpha_2, \beta_1, \beta_2 >0$, and generally $\alpha_1 \neq \alpha_2$ and $\beta_1 \neq \beta_2$
  Then the following hold:
  \begin{enumerate}[label=(\roman*)]
  \item The probability that the unit does not fail at time $t$ is given by
    \begin{equation}
      \exp \left( - \left(\frac{t}{\beta_1}\right)^{\alpha_1} - \left(\frac{t}{\beta_2}\right)^{\alpha_2} \right);
        \label{eq:event_weibull-1}
    \end{equation}
  \item FPTF is given by
    \begin{equation}
      1 -  \exp \left( - \left(\frac{t}{\beta_1}\right)^{\alpha_1} - \left(\frac{t}{\beta_2}\right)^{\alpha_2} \right);
    \end{equation}
  \item If $\alpha_1 = \alpha_2 =: \alpha$, the mean of FTPF is given by 
    \begin{equation}
      \left(\frac{1}{\beta_1^\alpha} +  \frac{1}{\beta_2^\alpha} \right)^{-1 / \alpha}
      \Gamma\left(1 + \frac{1}{\alpha}\right).
    \end{equation}
  \end{enumerate}
\end{proposition}

\begin{proof}
  Since (i) and (ii) are obvious, we prove only (iii).
  If $\alpha_1 = \alpha_2 =: \alpha >0$, then we have 
  \begin{equation}
    \label{eq:event_weibull-pr1}
    \int_0^\infty \exp \left( - \left(\frac{t}{\beta_1}\right)^{\alpha_1} - \left(\frac{t}{\beta_2}\right)^{\alpha_2} \right) \, dt
    = \int_0^\infty \exp \left( - \left( \frac{1}{\beta_1^\alpha} + \frac{1}{\beta_2^\alpha} \right) t^\alpha\right) \, dt.
  \end{equation}
  By letting $s: =(1/\beta_1^{\alpha} + 1/\beta_2^{\alpha}) t^\alpha$, Eq. \eqref{eq:event_weibull-pr1} becomes
\begin{equation}
  \label{eq:4}
  \frac{1}{\alpha} \left(\frac{1}{\beta_1^\alpha} + \frac{1}{\beta_2^\alpha}\right)^{-1/\alpha} 
  \int_0^\infty \exp (-s) s^{ 1/\alpha -1} \, ds = \left(\frac{1}{\beta_1^\alpha} + \frac{1}{\beta_2^\alpha}\right)^{-1/\alpha} \frac{1}{\alpha} \Gamma\left(\frac{1}{\alpha}\right).
\end{equation}
  \end{proof}

  \section{Conclusions}
\label{sec:conc}

The present study investigated catastrophic failure and cumulative damage models involving two different types of shocks.
As characteristic values of the cumulative damage models, we presented the distribution functions of the sum of damage, its mean, and the first passage of time to failure, for the case where the magnitudes of the events followed general and integer gamma distributions, under the assumption that the interarrival times between shocks followed exponential distributions.
Regarding the catastrophic failure models, we showed the probability that the unit did not fail at time t and the first passage of time to failure, as well as their respective means, for the case where the interarrival times between shocks followed general, integer gamma, and Weibull distributions.
Numerous previous studies were concerned with exponential distributions, but their results essentially could not be discussed in terms of two types of stochastic process due to the reproductive property; and many real-world cases involve combinations of such process types.
Investigations of general gamma and more general distributions would be needed to analyze such real-world examples, and thus represent interesting future challenges, but the present study offers a first step toward such investigations.

\section*{Acknowledgements}
The authors are deeply grateful to Prof.\ Dr.\ Toshio Nakagawa.
Without his guidance and persistent help, this paper would not have been possible.

\end{document}